\documentclass[11pt]{article}
\usepackage{setspace}   
\usepackage[toc,page]{appendix}
     \usepackage{amssymb}  
     \usepackage{amsthm}
     \usepackage{amsmath}
     
     \usepackage{mathtools}
     
     \usepackage{enumitem}
     
     \usepackage{pifont}

\newtheorem{theorem}{Theorem}
\newtheorem{corollary}[theorem]{Corollary}
\newtheorem{definition}[theorem]{Definition}

\newtheorem{construction}[theorem]{Construction}
\newtheorem{lemma}[theorem]{Lemma}
\newtheorem{observation}[theorem]{Observation}
\newtheorem{remark}[theorem]{Remark}
\newtheorem{aside}[theorem]{Aside}

\begin{document}
\newcommand{\al}{\alpha}
\newcommand{\be}{\beta}
\newcommand{\bul}{\bullet}
\newcommand{\lam}{\lambda}
\newcommand{\cF}{\mathcal{F}}
\newcommand{\Ga}{\Gamma}
\newcommand{\cR}{\mathcal{R}}
 \newcommand{\ib}{\textit{ibid.\,}}
\newcommand{\id}{\textit{id.}\,}
\newcommand{\mydel}[1]{}
\newcommand{\bbR}{\mathbbR}   
\newcommand{\cQ}{\mathcal{Q}}

\newcommand{\invtw}{\tfrac{1}{12}}
\newcommand{\myvspace}[1]{\vspace*{1mm}}
\newcommand{\dmjdel}[1]{}

\newcommand{\BB}{\scalebox{2}{$\bullet$}}

\title{A Non-constructive Proof of the Four Colour Theorem.} 

\author{D. M. Jackson and  L. B. Richmond,\\ 
Dept. of Combinatorics and Optimization, \\
University of Waterloo, Waterloo, Ontario,
Canada.}

\maketitle 

\begin{center}
{\bf Dedicated to the Memory of W.T. Tutte}. 
\end{center}

\begin{abstract} 
 The approach is through a singularity analysis of generating functions for 3- and 4-connected triangulations, asymptotic analysis,  properties of the ${{}_3F_2}$ hypergeometric series, and Tutte's enumerative work on planar maps and chromatic polynomials.   
\end{abstract}

\section{Background,  Approach and a set $Q$ of maps.} 
A planar map or, more briefly, a \textit{map}, is a 2-cell embedding of a connected graph in the plane.  Determining the least number of colours required for colouring the vertices of a map so that no edge joins vertices of the same colour has a history extending over some 170 years (Saaty \& Kainen \cite{Saaty}).  The \textit{Four Colour Conjecture} (in its dual form) asserted that every planar map is vertex 4-colourable.   The first proof, by Appel \&  Haken~\cite{Appel1} in 1977, is heavily case-analytic, through the graph-theoretic operations of discharging and reducibility.  A  shorter such proof was given by Robertson \textit{et \textit{al.}}~\cite{Robertson} in 1997. These proofs are \textit{constructive}, giving a polynomial-time algorithm for a vertex $4$-colouring.  
  Chromatic polynomials were introduced by Birkhoff \& Lewis~\cite{Birkhoff} who, with Tutte, believed they should assist in a proof of the Four Colour Problem.

\subsection{Preliminaries}
\begin{remark}\label{assertion1}  {\bf The Approach:}  
The Four Colour Problem is equivalent to the following assertion:   every 3-connected triangulation (defined below) is 4-colourable ({\it p.}289, Bondy \& Murty \cite{Bondy}).     
{\rm  This assertion will now be proved  through the  enumeration of maps.  In the proof all of the maps are rooted.}
\end{remark}

\begin{definition}\label{DEF1}
A triangulation with no loops or multiple-edges is said to be \textbf{3-connected}.  A \textbf{separating 3-cycle}  is a 3-cycle with at least one internal vertex and at least one external vertex.  A triangulation with no separating 3-cycle is said to be \textbf{4-connected} (or {\it simple}).   A $3$-connected triangulation with a separating 3-cycle is said to be \textbf{cyclically 3-connected.}
\end{definition}

Following Tutte~\cite{Tutte1},  $g(x)$ and $h(x)$ denote the generating functions for 3- and 4-connected triangulations. 
He supposed that triangulations have $2n+2$ faces and showed that if $g(x)= \sum_{n=1}^\infty g_n x^n$ then

\begin{equation}\label{GNSIM}  
(a) \quad g(x) = \sum_{n\ge1}  \tfrac{2} {(n+1)!}  \tfrac{(4n+1)!} {(3n+2)!} x^n 
\quad\mbox{and}\quad  \quad (b) \quad 
 g_n \sim \tfrac{1}{16} \sqrt{\tfrac{3}{2\pi}}  n^{-\tfrac{5}{2}}  (\tfrac{256}{27})^{n+1}.
 \end{equation}
 where (a) is from (2.2)  and  (5.11) of \cite{Tutte1},  and  (b) is from (8.1) of \cite{Tutte1} through Stirling's Formula.  Note that $g(x)-h(x)$ is the generating function for 3-connected triangulations with separating 3-cycles.  Tutte also showed that   $h_n = [x^{2n+2}]  h(x) \sim  \tfrac{1}{128}\sqrt{\tfrac{3}{\pi}} \,  n^{-\tfrac{5}{2}}  (\tfrac{27}{4})^{n+1}$    
from (8.1)~\cite{Tutte1}. 

\begin{remark}{\bf A passage between algebraic and analytic combinatorics.} 
This is afforded by:  
\begin{enumerate}  \label{Methodology}
\item [$(a)$]  the \textit{fundamental formula}   
$ [x^n] (1-x)^{-\al} = \frac{n^{\al -1}}{\Gamma(\al) }  \left( 1 +\mathcal{O}\left( \tfrac{1}{n}\right)   \right),$ where $\al \in \mathbb{C}  \backslash  \mathbb{Z}_{\le0}$, ({\it p.381}, Thm.VI.1,  Flajolet \& Sedgewick \cite{Flaj}),   $\Gamma(x)$ is the gamma function, $\Gamma(\tfrac{1}{2})=\sqrt{\pi}$, and $[x^n]$ is the coefficient extraction operator;  
\item [$(b)$]  Pringsheim's Theorem ({\it p.}240 \cite{Flaj}), namely:  
{\it If a series $f(z)$  is representable at the origin by a series expansion that has non-negative integer coefficients and radius of convergence $r$, then the point $z=r$ is a singularity of $f(z).$}   
\end{enumerate}
\end{remark}

\subsection{3- and 4-connected triangulations.}   
  Several of Tutte's results in \cite{Tutte1} and \cite{Tutte2}  will be used.  In his  census of planar triangulations Tutte included the 3- and 4-connected triangulations and determined their generating functions for those with $2n$ faces for $n\ge1$ (see (\ref{VFE})).
 He denoted their generating functions by $g(x)$ and $h(x)$, respectively, and included their asymptotics and radii of convergence, $r_g$ and $r_h$.  In view of this, the same notation is used in this paper, despite their having a small, yet significant, difference in initial conditions for $n\le2$.  The two notations are distinguished in the text by writing 
`Tutte's $g(x)$ and $h(x)$',   and the    `present paper's $g(x)$ and $h(x)$'.

\subsubsection{A combinatorial construction.} 
 \begin{construction}\label{T34CON} 
To every 3-connected triangulation $T$ with at least 4 vertices there corresponds a unique 4-connected triangulation $H$ such that $T$ is equal to faces of $H$ filled in with 3-connected triangulations inscribed in each of its internal faces and identifying the two boundaries. (See Tutte \cite{Tutte1}.)       
\end{construction}
 
 The $m$ internal faces of $H$ receive triangulations with $n_1, n_2, \ldots$  or $n_m$ internal faces.  Two edges of a face meeting at a vertex determine an internal angle, and a face of a triangulation has three internal angles.  If $E, F$ and $V$ denote the number of edges, faces and vertices in a triangulation, then $V - E + F = 2$ from Euler's  equation.  Thus $2V - 2E + 2F = 4$.  An internal angle is the intersection of two edges of an internal face so the number of internal angles is $2E$. The number of internal angles is also $3F$ so $2E = 3F.$  These give $2V = F + 4,$ so $F$ is even.   The number of triangulations with $V = n+2$ or  $F=2n$ or $  E = 3n$, where $n\ge 1$, are equal so counting with respect to edges, faces or vertices are equivalent problems.  Note that
 \begin{equation} \label{VFE}
 V = n+2, \quad  F = 2n, \quad  E = 3n \quad \mbox{where} \quad n\ge1.
 \end{equation}

Let $g_n$ and $h_n$ denote, respectively,  the number of 3-connected and 4-connected triangulations with $2n+1$ internal faces.  In this context the map consisting of just one triangle is regarded as 3-connected but not 4-connected.   
Tutte also wrote $g_n = \psi_{n,0}$, where $\psi_{n,0}$ is the number of triangles with $k=m+3$ external edges so $3 = m+3$ for triangulations and $m=0$ ((1.3) Tutte \cite{Tutte1}).  Furthermore, the number of internal edges is denoted by $r$ and, according to \textit{(1.4)} of Tutte \cite{Tutte1}, $r = 3n +  m = 3n$, so the number of edges in a triangulation is $3(n+1)$.

\subsubsection{A divergence from Tutte's convention for $g(x)$ and $h(x)$.}
There is a single, but significant, difference between {\it Tutte's notation} and the {\it present paper's notation} concerning $g(x)$ and $h(x)$: Tutte defined $g(x)$ by
$g(x) = 1 + x +3 x^2 + 13 x^3 + \cdots,$ 
whereas in the present paper $g(x)$ is defined by
$g(x) = x +3 x^2 + 13 x^3 + \cdots.$       
As a consequence, it will be seen that the $g(x)$ and $h(x)$ (of the present paper) are related by $h(g(x)) = g(x) - x$,   and that this  identity, which is vital to this approach,  does not hold in the case of Tutte's $g(x)$ and $h(x)$.  Tutte also indicated that 
$h(x) = x^3 + \cdots$ since he does not count the triangulation with 2 faces as a 4-connected triangulation.  The triangulation with one internal face cannot be constructed  through Construction \ref{T34CON}, so (as also explained in  \cite{Cor1RRW}) 
\begin{equation}\label{EQU2}
h(g(x)) = g(x) - x.
\end{equation}
This accounts for the term $-x$ on the righthand side of this expression.  These  $g(x)$ and $h(x)$ (of the present paper) differ at $x=0$  from those of Tutte's. This does not affect their asymptotic limits.

\subsubsection{The set $\mathcal{Q}$ of maps and the proof strategy.}\label{ApStrat}  
A conjectural enumerative restatement of the Four Colour Problem.

\begin{observation}\label{Obs}
   (\textit{para.~2,  p.142, Richmond {\it et al.} }\cite{Cor1RRW})  {\it If the fraction of 4-colourable triangulations with $2n$ faces (or $n+2$ vertices) is not exponentially small then the 4-Colour Theorem  follows.}  
\end{observation}

 Theorem~\ref{thm1} will show that the condition in  Observation \ref{Obs} may be satisfied by exhibiting such a subset of maps.   The latter will be denoted by $\mathcal Q$.   For this purpose a 3-connected triangulation will be said to be {\it cyclically separable} if it has  a separating 3-cycle (see Def. \ref{DEF1}).

 \begin{definition}  
 $\mathcal{Q}$ is the set of all rooted 3-connected triangulations $T$ having a rooted 3-connected triangulation $T_1$ with $V(T_1)\ge4$  vertices inside a 3-cycle $C$ of $T$ and a rooted 3-connected triangulation $T_2$ with $V(T_2)\ge4$ vertices outside $C$ so that $C$ is a separating $3$-cycle of $T$. (Note that we do not suppose that $T_1$ and $T_2$ have a 3-cycle.)
\end{definition}

It is noted that $(g(x) - h(x))^2$ is the generating function for the elements of $\cQ$ with $2n+2$ faces, and that $T_1$ and $T_2$ are not required  to have a  separating 3-cycle.  Note also that, by definition,  $T_1$ and $T_2$ are 3-connected and also that each has  at least $4$ vertices.  Thus $T_1$ and $T_2$ are elements of $\mathcal{Q}$ by definition.  It will be proved by induction that the elements of $\mathcal{Q}$ are 4-colourable (see Lemma \ref{QF4C}).
  
If a triangulation $T$ has a separating 3-cycle $C$ with the triangulation $L$ inserted into the interior of $C$ then $T$ is said to \textit{contain a copy} of $L$.  Since $L$ and the rest of $T$, namely, $T-L$, intersect in a 3-cycle, which is the complete graph $K_3$ with three vertices, then,  from ({\it p.230}, Thm.~IX.27, Tutte \cite{Tutte2}), the associated chromatic polynomials are related by  
\begin{equation}\label{EQU1}
\lam (\lam-1)(\lam-2)\cdot  P(T,\lam) = P(L,\lam)\cdot P(T-L,\lam). 
\end{equation} 
Thus if $L$ cannot be 4-coloured, so $P(L,4)=0$, then $T$ cannot be 4-coloured.

\section{The set $\mathcal{Q}$ of maps and their 4-colourability.} 

\subsection{A relation between the generating functions $g(x)$ and $h(x)$.}   
Thus, from  (\ref{GNSIM}), the radius of convergence of $g(x)$ is  $r_g = \tfrac{27}{256}$  and, because of the factor $n^{-5/2}$, it follows that  $g(r_g)$ converges.   The following corollary (Cor.~1. Richmond \textit{et al.}~\cite{Cor1RRW})  is now re-proved through the methodology of this paper. 

\begin{corollary}\label{ctLrc}
 If there is one 3-connected triangulation $L$ which cannot be 4-coloured then the radius of convergence for the 4-colourable  3-connected triangulations is strictly greater than that for 3-connected triangulations.   
 \end{corollary} 

 \begin{aside}\label{rem2}
 Note that Corollary~\ref{ctLrc} implies that if a positive fraction of the 3-connected triangulations can be 4-coloured then the  4-colour theorem holds for triangulations.
 \end{aside}
 
\begin{proof}
 Let  $[x^{2i+1}] g_L(x)$ be the number of 3-connected triangulations not containing a copy of a 3-connected triangulation $L$ with $2i+1$ interior faces.    A similar argument to the one for (\ref{EQU2}) shows that  
 \begin{equation}\label{EQU3}
 h(g_L(x)) = g_L(x)-x. 
 \end{equation}
Note also  that
 \begin{equation}\label{gLg2}
g_L(x) = g(x) - x^{2i+1}. 
\end{equation}
   The coefficients of $g(x)$ and $g_L(x)$ are non-negative so, by Pringsheim's Theorem (see Remark~\ref{Methodology}(b)),  
     $r_g$ and $r_L$, the radii of convergence of $g(x)$ and $g_L(x)$, are  singularities of  $g(x)$ and $g_L(x)$, 
   respectively. 
(Note that a 4-connected triangulation cannot contain a copy of a triangulation $T$ since it would then have a separating 3-cycle, and therefore could not be 4-connected.)   In view of Pringsheim's Theorem,  only $x\ge0$ needs to be considered for radii of convergence arguments. 
 Let $r_h$  denote the radius of convergence of $h(x)$.   If $0\le x < r_g$ then $h(g(x))$ is analytic by (\ref{EQU2}), so $r_h \ge g(r_g)$.  Also  $g(r_g)$ converges because of the factor $n^{-5/2}$ in (\ref{GNSIM}(b)).  If $r_h < g(r_g)$ then, from (\ref{EQU2}),  it follows that  $h(g(x))$ and $g(x)$ have a singularity for $0 < x < r_g$, which is not possible,  so $r_h \ge g(r_g)$.  Now suppose that $r_h > g(r_g)$.  Then $h(g(x))$ is analytic at $x=r_g$. 
 But $h(g(x)) = g(x) - x$ from (\ref{EQU2}) so $g(x) - x$ is analytic at $x=r_g$.  Thus $g(x)$ is analytic at $x=r_g$.  
This contradicts Pringsheim's Theorem, so $r_h$ is not strictly greater than $g(r_g)$.  Hence 
\begin{equation}\label{rhgrg}
r_h = g(r_g).  
\end{equation}
 In addition, $r_L$, the radius of convergence of $g_L(x)$,  satisfies equation (\ref{EQU3}), and note that the coefficients of $g_L(x)$ are bounded by those of $g(x)$ since  every map counted by $g_L(x)$ is counted by $g(x)$ and so $r_L\ge r_g$.
 Now suppose that $r_L = r_g$.  
From~(\ref{rhgrg}),  a similar argument gives $r_L = g_L(r_L)$, and if $r_L= r_g$ then $g(r_g) = g_L(r_g).$    The coefficients of $g_L(x)$ are less than or equal to those of  $g(x)$.  But, from (\ref{gLg2}),  
 $[x^{2i+1}] g_L(x) = -1 + [x^{2i+1}] g(x)$ since $L$ is not counted by $g_L(x)$.  Thus  $g(r_g) < g_L(r_g)$,  contradicting $g(r_g) = g_L(r_g)$.   The supposition that $r_L = r_g$  leads to a contradiction, so $r_L > r_g,$  concluding the proof.   
 \end{proof}

\subsection{ Asymptotic forms for $g_n$.} \label{G34crt}   
 From \cite{Cor1RRW}, let 
$f_1(x) = 2\sum_{n\ge1} \tfrac{(4n-3)!}{n! (3n-1)!} x^n = \sum_{n\ge1}  \tfrac{1}{n!} f_{1,n} x^n$ 
\; \quad\mbox{so} \; \quad
$\tfrac{f_{1,n+1}} {f_{1,n}} = \tfrac{(n + \tfrac{1}{4}) (n - \tfrac{1}{4}) (n - \tfrac{1}{2})}  {(n + \tfrac{2}{3}) (n + \tfrac{1}{3}) }
\, \tfrac{4^4}{3^3}.$  Let  $f(x)$  be defined  by
$f(x) := f_1(\tfrac{27}{256} x) = 
 \tfrac{3}{4} \left(   {}_3F_2 (   \tfrac{1}{4},   - \tfrac{1}{4},  -\tfrac{1}{2};  \tfrac{2}{3},   \tfrac{1}{3} ; x  )  -1  \right),$
a  generalised hypergeometric function.
 Thus $f(x)$ has the same singularities as a  $ {}_3F_2$;  that is, at $x\in \{ 0, 1, \infty \}$ (Norlund \cite{Norlund}), but not at $x=0$ since, in the present paper's definition of  $g(x)$, 
$g(x) = x +2\sum_{n\ge 2} \tfrac{(4n+1)!} {(n+1)!(3n+2)!} \, x^n.$
Note that  $g_n>0$  for $n\ge1$, so from (\ref{GNSIM}(b)),  
\begin{equation}\label{gxconv}
g(x) \text{ converges at } x=\tfrac{27}{256}.
\end{equation}
Norlund \cite{Norlund} showed that the ${}_3F_2$ defining $f(x)$ has only one singularity on its circle of convergence, namely at $x= \tfrac{27}{256},$  and that, near $x= \tfrac{27}{256}$,  $g(x)$ has  the form 
\begin{equation}\label{gg1g2}
 g(\tfrac{27}{256} -x) = (x - \tfrac{27}{256})^{3/2} g_1(x) + g_2(x)
 \end{equation} 
 where $g_1(x)$ and $g_2(x)$ are  analytic at $x=\tfrac{27}{256}$.  
(In fact, $g_1(x)$ and $g_2(x)$ are analytic in the  finite complex plane, but not in the neighbourhood of $\infty$.)
Now 
\begin{equation} \label{g1g2A1AB}
 g_1(x) = B + \mathcal{O}\left( | \tfrac{27}{256} -x  |   \right)  \text{near}\, x = \tfrac{27}{256}, \quad {and} \quad   
g_2 (x) = A + A_1 (\tfrac{27}{256} -x) + \mathcal{O}\left( | \tfrac{27}{256} -x  |^2   \right),
\end{equation}  
where $A,  A_1$ and $B$ are constants (and shall be seen to be positive). 
  
  Singularity analysis may be used with Norland's expansion for $g(x)$ near $x= \tfrac{27}{256}$ and  also the {\it fundamental formula} of Remark \ref{Methodology}(a) to determine the asymptotic behaviour of $g_n$.  
 A contour of integration is chosen for estimating $[x^n] g(x)$.  This contour  encloses a $\triangle$-region  and depends upon $n$ and is  the one chosen for singularity analysis in Flajolet \& Sedgewick \cite{Flaj}.
  After circling $x=\tfrac{27}{256}$ the boundary of the $\triangle$-region continues in straight lines to the circle $|x| = \tfrac{27}{256}$ + $\delta$ where $\delta = \tfrac{1}{n}$.   
  Corollary~3 of Theorem~1 (\textit{p.}225, Flajolet \& Odlyzko \cite{FlajOdl}) states that if 
$f(x) = \sum_{j=0}^J  c_j \left(1 -\tfrac{256}{27}x\right)^{\al_j} +  \mathcal{O}\left(\vert 1 - \tfrac{256}{27}x\vert^C    \right)$
where $J$ is a non-negative integer and 
where $ \mathbb{R}(\al_0)    \le    \mathbb{R}(\al_1)  \le \cdots   \le \mathbb{R}(\al_n)  < M$ where $M$ is finite and $M>J$ (this is an expansion of Norlund-type with a singularity at $x=\tfrac{27}{256}$), then as $n \rightarrow \infty$, 
$$[x^n] f(x) = \sum_{j=0}^J  d_j  {n- \al_j-1 \choose n} + \mathcal{O} (n^{-C-1}),$$ 
where $C > \al_J$ (and $C, \al_J$ and $M$ are finite).   And so ((2.27c) Flajolet \& Odlyzko \cite{FlajOdl}) it follows that  
$$[x^n] f(x) = \sum_{j=0}^J  d_j \, n^{-\al_j-1}  + \mathcal{O}(n^{-C-1}).$$ 
 Only the  $j=0$ term will be used.  The expansion of $f(x)$ is in powers of $1 -\tfrac{256}{27}x$, rather than in powers of $1-x$, so 
$g_n = [x^n]g(x) \sim B\left(\tfrac{256}{27}\right)^n  n^{-5/2} \left(1 + \mathcal{O}(\tfrac{1}{n}) \right).$  
Thus, from (\ref{GNSIM}(b)),   
\begin{equation}\label{EQU4}
B = \tfrac{16}{27} \sqrt{\tfrac{3}{2\pi}}. 
\end{equation} 

 In the following argument it should be noted that:     
(a) terms containing powers of $x - \frac{27}{256}$ greater than $\tfrac{3}{2}$ are irrelevant according to singularity analysis, and are immediately neglected;  and
(b)  the generating function for triangulations with separating 3-cycles (see Def. \ref{DEF1}) is $g(x) - h(x)$. 
Thus, from  (\ref{gg1g2}), 
$$(g(\tfrac{27}{256} - x))^2 = (g_1(x))^2  (\tfrac{27}{256} -x)^3  + 2 g_1(x) g_2(x)  (\tfrac{27}{256} -x)^{3/2} + (g_2(x))^2 + \mathcal{O}( |x -\tfrac{27}{256} |^2)$$
  near $x=\frac{27}{256}$, so $(g(\tfrac{27}{256} - x))^2 =  2 g_1(x) g_2(x)   (\tfrac{27}{256} -x)^{3/2}$ plus a function that is analytic in the finite complex plane.  Then, from (\ref{g1g2A1AB}),  
  
\begin{equation}\label{ABequg2g} 
[x^n](g(\tfrac{27}{256} - x))^2 = 2 A \cdot B \,  [x^n]  (\tfrac{27}{256} -x)^{3/2}  (1+\mathcal{O}(\tfrac{1}{n})) \sim   3\,\tfrac{A \cdot B}{2\sqrt{\pi}} n^{-5/2}  \left(  \tfrac{256}{27} \right)^n  \left( 1 + \mathcal{O}{(n^{-1})}  \right).  
\end{equation}  
From (\ref{EQU4}),  and   
 (\ref{GNSIM}(a)),   
 for  $n\ge2$    
 \begin{equation}\label{EQN7}
A = g(\tfrac{27}{256}) =   \tfrac{27}{256} + 2\sum_{n\ge2}  g_n \left(\tfrac{27}{256}\right)^{n+1}, 
\end{equation}
which, from (\ref{gxconv}), converges. Thus $A$  is  positive (since $g_n \ge0$).  Then the fraction of the 3-connected triangulations with $2n$ faces that are elements of $\mathcal{Q}$ is asymptotic to a positive constant as $n\rightarrow \infty$. This is obtained  
from (\ref{ABequg2g}) and  (\ref{GNSIM}(b)),
\begin{equation} \label{EQN8}
\frac{[x^n] \, (g(x))^2}{[x^n] \, g(x)} \sim    \tfrac {27} {2}  \sqrt{\tfrac{3}{2}} \cdot A\cdot B
\end{equation}
where $A\cdot B >0$  from (\ref{EQU4}) and (\ref{EQN7}).   
 This proves the following. 

\begin{theorem}\label{thm1}
The number of elements of $\mathcal{Q}$ with $2n$ faces is asymptotic to a positive fraction of the number of 3-connected triangulations with $2n$ faces as $n\rightarrow\infty$.
\end{theorem}

 \subsection{4-colourability of maps in $\mathcal{Q}$.} 
\begin{lemma}\label{QF4C}
The elements of $\mathcal{Q}$ with $2n$ faces are $4$-colourable.
\end{lemma}
\begin{proof} 
Note that $V(T) = V(T_1)  + V(T_2)  -3$ since the vertices of $T_1$ and $T_2$  in a separating 3-cycle are counted twice with $T_1$ and also $T_2$.  
  The following induction hypothesis is introduced:  
 \textit{All $3$-connected triangulations in $\mathcal{Q}$ with $4\le V(T) \le n-1$ are 4-colourable.} 
 
Consider an element, $T$,  of  $\mathcal{Q}$ with $n$ vertices.  From the definition of $\mathcal{Q}$, $T$ is divided by a separating 3-cycle into triangulations $T_1$ and $T_2$.   
These may or may not have separating 3-cycles.  Let these have $n_1$ and $n_2$ vertices, respectively.  Then there are $n_1 + n_2 - 3$ vertices in $T$, so 
\begin{eqnarray}\label{EQU4435} 
V(T) = V(T_1) + V(T_2) - 3
\end{eqnarray}
 since the vertices of $T_1$ in the separating 3-cycle $C$  are counted twice, once with $V(T_1)$ and once with  $V(T_2)$.  Thus $V(T_1) = V(T) +3 - V(T_2) \le V(T) + 3 - 4 \le V(T) - 1 < V(T).$  Similarly, $V(T_2) \le V(T) -1$.   Thus, from the Induction Hypothesis, $T_1$ and $T_2$ are 4-colourable.  From (\ref{EQU1})    
\begin{eqnarray}\label{EQU9}
\lam(\lam -1) (\lam-2)\cdot P(T,\lam) = P(T_1,\lam) \cdot  P(T_2,\lam).
\end{eqnarray} 

From the induction hypothesis $P(T_1,\lam) > 0$ and $P(T_2,\lam) > 0$ since $T_1$ and $T_2$ are 4-colourable, having fewer vertices than $T$. Thus $P(T,4) >0$ and $T$ is 4-colourable.   A 3-connected triangulation $U_1$ has fewer vertices than a 3-connected triangulation $U_2$ if and only if it has fewer faces than $U_2$  (see  (\ref{VFE})).
    
The base case may be established as follows.  
The chromatic polynomial  of a triangulation with $4$ vertices is the chromatic polynomial $P(K_4,\lam)$ of the complete graph $K_4$, which is $\lam (\lam-1)  (\lam-2) (\lam-3)$.
The smallest element  $S$ of the set $\mathcal{Q}$ is   two $K_4$'s  intersecting in a 3-cycle (the latter is a $K_3$).  Thus
$\lam (\lam-1)  (\lam-2) \cdot P(S,\lam) = P(K_4,\lam) ^2$  from (\ref{EQU9}) so 
 $P(S,\lam) = \lam (\lam-1)  (\lam-2) (\lam-3)^2$ and $P(S,4) = 24 > 0.$  The smallest elements in $\mathcal{Q}$ constructed from  two $K_4$'s are therefore 4-colourable.  Moreover, the number of triangulations with $4$ vertices  may be deduced from Tutte's results.   Also $m=0$ in (1.3) and (1.4) of Tutte~\cite{Tutte1}  so $n+2=4$ and $n=2.$  The number of smallest elements of 
 $\mathcal{Q}$ is $g_2^2 =  9$  from (\ref{GNSIM}(a)).   From (\ref{EQU4435}), the smallest elements  of the base case of the induction have $4 + 4 - 3 = 5$ vertices,  and each is 4-colourable.  They form a set of $9$ 4-colourable triangulations (which is the base case for the induction).  Thus,  from the inductive hypothesis,  all of the elements of $\mathcal{Q}$ are 4-colourable.   
 \end{proof}
 
From the asymptotic formula (\ref{EQN8}) for $[x^n] (g(x))^2/ [x^n]g(x)$ it follows  that a positive fraction of these triangulations are 4-colourable.   See also  Observation~{\ref{Obs}}.   From Aside \ref{rem2}, this completes the proof of the following theorem.

\begin{theorem}\label{thm2}
 The elements of $\mathcal{Q}$ are 4-colourable.
 \end{theorem}    

Equation (\ref{EQN8}) states that the number of elements of $\mathcal{Q}$ with $2n$ faces is a positive fraction of the number of 3-connected triangulations with $2n$ faces.  Lemma~\ref{QF4C} states that the elements of $Q$ with $2n$ faces are $4$-colourable.  Thus  a positive fraction of the 3-connected triangulations are 4-colourable.    Aside~\ref{rem2} states that if a positive fraction of the 3-connected triangulations are 4-colourable then the 4-Colour Theorem holds for  triangulations.  From Remark \ref{assertion1} the 4-Colour Theorem holds.   \\[1pt]

\noindent {\bf Acknowledgement}.
We thank those of our colleagues who offered insightful comments.

\end{document}